\documentclass[12pt,a4paper,final]{amsproc}

\usepackage{amssymb}
\usepackage{amscd}
\usepackage[all]{xy}

\usepackage[cp850]{inputenc}
\usepackage[mathscr]{eucal}
\usepackage[notcite,notref]{showkeys}
\begin{document}
\title[The Twisted Ideals]{The Twisted Hilbert Space Ideals}
\author[Cabello \& Garc\'ia]{F\'{e}lix Cabello S\'{a}nchez and Ricardo Garc\'ia}
\thanks{Research supported in part by MINCIN Project PID2019-103961GB-C21 and Junta de Extremadura project IB20038.}


\address{Departamento
de Matem\'{a}ticas and IMUEx, Universidad de Extremadura}
\address{Avenida de Elvas,
06071 Badajoz, Spain.} \email{fcabello@unex.es, rgarcia@unex.es}
\noindent {\footnotesize Version \today}

\bigskip

\bigskip

 \maketitle
\theoremstyle{plain}
\newtheorem{problem}{Problem}
\newtheorem{theorem}{Theorem}
\newtheorem{proposition}{Proposition}
\newtheorem{corollary}{Corollary}
\newtheorem{lemma}{Lemma}
\theoremstyle{remark}
\newtheorem{definition}{Definition}
\newtheorem{example}{Example}
\newtheorem{remark}{Remark}
\newtheorem*{question}{Question}

\newcommand{\R}{\mathbb{R}}
\newcommand{\N}{\mathbb{N}}
\newcommand{\K}{\mathbb{K}}
\newcommand{\e}{\varepsilon}


\newcommand{\Hi}{\mathscr H}
\newcommand{\BH}{B(\mathscr H)}
\newcommand{\KH}{\mathcal K(\mathscr H)}
\newcommand{\LH}{\mathcal L(\mathscr H)}
\newcommand{\EH}{\mathcal E(\mathscr H)}
\newcommand{\FH}{\mathcal F(\mathscr H)}
\newcommand{\QH}{\mathcal Q(\mathscr H)}
\newcommand{\Loo}{\ensuremath{{L_\infty }}}

\newcommand{\EL}{\ensuremath{\Ext_{L^\infty}}}

\newcommand{\To}{\ensuremath{\longrightarrow}}

\def\Ext{\operatorname{Ext}}
\def\Hom{\operatorname{Hom}}
\def\PB{\operatorname{PB}}
\def\PO{\operatorname{PO}}
\def\El{\operatorname{Ext}}
\def\sign{\operatorname{sign}}
\def\dim{\operatorname{dim}}
\def\supp{\operatorname{supp}}
\def\co{\operatorname{co}}
\def\Bil{\operatorname{\mathfrak B}}
\def\Bal{\operatorname{Bal}}
\def\dist{\operatorname{dist}}
\def\M{\operatorname{M}}
\bibliographystyle{plain}

\begin{abstract}
We study those operators on a Hilbert space that can be lifted / extended to any twisted Hilbert space. We prove that these form an ideal of operators which contains all the Schatten classes. We characterize those multiplication operators on $\ell_p$ that are liftable / extensible through centralizers.
\end{abstract}

\section{The object of study}

A twisted Hilbert space is a short exact sequence of (quasi) Banach spaces and operators
\begin{equation}\label{kth}
\begin{CD}
0@>>> \mathscr K @>\imath >> \mathscr T @>\pi>> \mathscr H @>>> 0
\end{CD}
\end{equation}
in which $\mathscr H$ and $\mathscr K$ are Hilbert spaces. Since ``exact'' means that the kernel of each arrow agrees with the image on the preceding, we see that $\imath$ is an isomorphic embedding and that $\pi$ induces an isomorphism between $\mathscr H$ and   $\mathscr T/\imath[\mathscr K]$. Thus, less technically, we can regard $\mathscr T$ as a (quasi) Banach space containing $ \mathscr K$ as a closed subspace in such a way that the corresponding quotient is  $\mathscr H$. This already implies that $\mathscr T$ is (isomorphic to) a Banach space, i.e., its quasinorm is equivalent to a convex norm; see \cite[Theorem 4.9 plus Theorem 4.10]{k} or \cite[Theorem 4.1 plus Theorem 6.2]{kaltconv}.

Being clear that $ \mathscr K$ is complemented in  $\mathscr T$ if and only if $\mathscr T$ is itself isomorphic to a Hilbert space (namely, to $ \mathscr K\oplus_2  \mathscr H$), the first thing one must know on the subject is that nontrivial twisted Hilbert spaces do exist. We refer the reader to \cite{elp, kp} for the early constructions and to \cite{cg}, \cite[Chapter 16]{bl} and the references therein for further developments. Readers who are not familiar with any of those references can stop here.

From now on all (quasi) Banach spaces are assumed to be complex and all Hilbert spaces separable. Considering nonseparable spaces would not lead to any new idea and would result in a number of irritating complications. 

\begin{definition}
An operator $f\in B(\mathscr H)$ is said to be liftable if for each twisted Hilbert space (\ref{kth}) there is an operator $F: \mathscr H \To \mathscr T$ such that $f=\pi\,F$. 

Analogously, $g\in B(\mathscr K)$ is said to be extensible if for each twisted Hilbert space (\ref{kth}) there is an operator $G: \mathscr T \To \mathscr K$ such that $G\,\imath=g$. 
\end{definition}

As the reader may guess, the key point of the definition is the quantifier ``for each''. In all what follows we fix a separable Hilbert space $\mathscr{H}$ and we write $\mathcal L(\mathscr H)$ and $\mathcal E(\mathscr H)$ for the set of liftable and extensible operators on $\mathscr H$ respectively.

\section{The ideal of liftable operators}
Let us begin with the elementary observation that operators in the Hilbert-Schmidt class $\mathcal S_2(\mathscr H)$ factorize through $\ell_1$ and, by the lifting property of $\ell_1$ and the fact that twisted Hilbert spaces are (isomorphic to) Banach spaces, are liftable:
$$
\xymatrixrowsep{1pc}
\xymatrixcolsep{3.5pc}
\xymatrix{
&& \mathscr T\ar[dd]^\pi\\
& \ell_1 \ar[ur]^{H} \ar[rd]^h\\
\mathscr H \ar[rr]^f \ar[ur]^g & & \mathscr H
}
$$
Recall that $\mathcal I\subset B(\mathscr H)$ is an ideal if $afb\in\mathcal I$ whenever $f\in\mathcal I$ and $a,b\in B(\mathscr H)$, in which case $\mathcal{I}$ is a linear subspace, closed under taking (Hilbert space) adjoints. By classical results of Calkin \cite[Theorem 1.6]{calkin}
{\em proper} ideals of $B(\mathscr H)$ are in correspondence with symmetric, solid linear subspaces of $c_0$: indeed, if $X\subset c_0$ is such, one can define the ideal $\mathcal I_X(\mathscr H)$ of those (necessarily compact) operators  whose singular numbers belong to $X$;
conversely, if $\mathcal I \subset B(\Hi)$ is a proper ideal and one fixes an orthonormal basis $(e_n)$ for $\mathscr H$, then the sequences $(s_n)$ for which the diagonal operator $\sum_n s_ne_n\otimes e_n$ belongs to $ \mathcal I$ form a symmetric, solid linear subspace of $c_0$.

\begin{lemma}\label{lem:isanideal} $\LH$ is an ideal of compact operators.
\end{lemma}

\begin{proof}
Let us first remark that if $f\in\mathcal L(\mathscr H)$ and $a\in B(\mathscr H)$, then $fa$ is liftable for if $F$ is a lifting of $f$, then $Fa$ is a lifting of $fa$. To check that $af$ is liftable too, take any twisted Hilbert space as in (\ref{kth}) and form the pullback with $a$ to get the commutative diagram
\begin{equation*}
\begin{CD}
0@>>> \mathscr K @>\imath >> \mathscr T @>\pi>> \mathscr H @>>> 0\\
 && @| @AA\overline{a}A @AAaA\\
0@>>> \mathscr K @> \overline{\imath}>> \PB @> \overline{\pi}>> \mathscr H @>>> 0
\end{CD}
\end{equation*}
Here, the {\em pullback} space is $\PB=\{(z,x)\in \mathscr T \times\mathscr H: \pi(z)=a(x)\}$ and the operators $\overline{a}, \overline{\pi}$ are the restrictions of the projections of  $\mathscr T \times\mathscr H$ onto $\mathscr T$ and $\mathscr H$ respectively. Finally, $\overline{\imath}(y)=(\imath(y),0)$ for $y\in \mathscr K$ and the lower sequence is exact.

Now, if $F$ is a lifting of $f$ to $\PB$ then $\overline{a}\, F$ is a lifting of $af$ to $\mathscr T$. To prove that every liftable operator is compact it suffices to see that $\LH$ is {proper}, which is obvious from the existence of nontrivial sequences (\ref{kth}): the identity of $\mathscr H$ cannot be lifted to $\mathscr T$.
\end{proof}

\begin{corollary}
An operator is liftable if and only if is extensible. 
\end{corollary}

\begin{proof}
As $\mathscr T$ is a Banach space elementary considerations 
on the relationships between short exact sequences of the form (\ref{kth}) and their (Banach space) adjoints
$$
\xymatrix{
0 \ar[r] & \mathscr H'\ar[r]^{\pi'} &  \mathscr T' \ar[r]^{\imath'} &  \mathscr K' \ar[r] & 0
}
$$
reveal that an operator is liftable if and only its Banach space adjoint is extensible, in particular $\EH$ is a proper ideal of $B(\mathscr H)$ for any Hilbert space $\mathscr H$. But a {\em compact} operator on a Hilbert space is liftable or extensible if and only if its Banach space adjoint is: let $f=\sum_{n}s_n y_n\otimes x_n$ be a Schmidt expansion of a compact operator on $\mathscr H$. We can use the fact that $(x_n)_{n\geq 1}$ is an orthonormal basis of $\Hi$ to define a surjective isometry $I:\Hi\To\Hi'$ sending $x_n$ to $\langle -|x_n\rangle$, that is, $\langle I(\xi), \eta\rangle=\sum_n \xi_n\eta_n$, where $\xi=\sum_n \xi_n x_n$ and $\eta=\sum_n \eta_n x_n$. The basis $(y_n)$ gives another isometry $J:\Hi\To\Hi'$ sending $y_n$ to $\langle -|y_n\rangle$. Since $f(y_n)=s_n x_n$ one has $f'( \langle -|x_n\rangle ) = s_n \langle -|y_n\rangle$, that is, the following square is commutative:
$$
\xymatrixcolsep{3.5pc}
\xymatrix{
\Hi \ar[r]^f \ar[d]_J & \Hi \ar[d]^I\\
\Hi' \ar[r]^{f'} & \Hi' 
}
$$
Nothing more to add.
\end{proof}

In order to gain a deeper understanding of $\LH$ we need the notion of a quasilinear map. A homogeneous mapping $\phi: X\To Y$ acting between quasinormed spaces is said to be quasilinear if it obeys an estimate
$$
\|\phi(x+y)-\phi(x)-\phi(y)\|\leq Q(\|x\|+\|y\|)
$$
for some constant $Q$ and all $x,y\in X$. The least constant for which the preceding inequiality holds is called the quasilinearity constant of $\phi$ and is denoted by $Q(\phi)$. The space of all quasilinear maps from $X$ to $Y$ is denoted by $\mathcal Q(X,Y)$, or just $\mathcal Q(X)$ when $Y=X$. 

Given a quasilinear map $\phi:X\To Y $ one can construct the (twisted sum) space  $Y\oplus_\phi X$ which is just the direct sum space $Y\oplus X$ equipped with the quasinorm $\|(y,x)\|_\phi=\|y-\phi(x)\|+\|x\|$. The map $y\longmapsto (y,0)$ is an isometric embedding of $Y$ into $Y\oplus_\phi X$ and the map $(y,x)\longmapsto x$ takes the unit ball of $Y\oplus_\phi X$ onto that of $X$ so that we have an exact sequence
\begin{equation}\label{eq:yyfxx}
\begin{CD}
0@>>> Y @>\imath >> Y\oplus_\phi X @>\pi>> X @>>> 0
\end{CD}
\end{equation}
called, with good reason, the sequence generated by $\phi$. The preceding considerations imply that $Y\oplus_\phi X$ is complete (that is, a quasi Banach space) if $X$ and $Y$ are. 
The relevant point of this discussion is that any exact sequence of quasi Banach spaces
$$
\begin{CD}
0@>>> Y @>\jmath >> Z @>\varpi>> X @>>> 0
\end{CD}
$$
is {\em equivalent} to one of the form (\ref{eq:yyfxx}) in the sense that there is a commutative diagram
$$
\xymatrixrowsep{1pc}
\xymatrixcolsep{3pc}
\xymatrix{
&& Z \ar[dd]^u \ar[rd]^\varpi &&\\
0 \ar[r] & Y \ar[ur]^{\jmath} \ar[dr]_{\imath} & & X \ar[r] & 0\\ 
&& Y\oplus_\phi X \ar[ur]_{\pi} &&
}
$$
in which $u$ is an isomorphism.

Now suppose we intend to lift a given operator $f:V\To X$ to $Y\oplus_\phi X $. Each lifting $F:V\To Y\oplus_\phi X$ has the form $F(v)=(L(v), f(v))$, where $L:V\To Y$ is a linear map, not necessarily continuous. Since
$
\|F(v)\|_\phi = \|L(v)-\phi(f(v))\|+\|f(v)\|
$
 we see that $F$
is bounded if and only if there is a constant $K$ such that
\begin{equation}\label{eq:L-phi u}
\|L(v)-\phi(f(v))\|\leq K\|v\|
\end{equation}
for all $v\in V$. From now on we shall use the following notation: given  a homogeneous map $h:X\To Y$, acting between quasinormed spaces, we put
$$
\|h\|=\|h:X\To Y\|= \sup_{\|x\|\leq 1}\|h(x)\|
$$
and we say that $h$ is bounded if $\|h\|<\infty$. In this way the inequality (\ref{eq:L-phi u}) means $\|L-\phi\circ f\|\leq K$.
Similarly, an operator $f:Y\To V$ extends to $Y\oplus_\phi X $ if and only if there is a linear map $L:X\To V$ such that $\|L-f\circ\phi\|<\infty$.
Since all separable Hilbert spaces are isometrically isomorphic and all twisted Hilbert spaces arise from quasilinear maps we have proved:

\begin{lemma}\label{lem:lif-char}
For an operator $f\in B(\mathscr H)$ the following conditions are equivalent:
\begin{itemize}
\item $f$ is liftable or extensible.
\item For every $\phi\in\mathcal Q(\mathscr H)$ there is a linear map $L$ on $\mathscr H$ such that $\|L-\phi\circ f\|<\infty$.
\item For every $\phi\in\mathcal Q(\mathscr H)$ there is a linear map $L$ on $\mathscr H$ such that $\|L-f\circ	\phi\|<\infty$.
\end{itemize}
\end{lemma}
Let $p\in(0,\infty)$. The Schatten
class $\mathcal{S}_p$ consists of those operators on $\Hi$ whose sequence of singular numbers $(s_n(f))$ belongs to $\ell_p$. It is a quasi Banach
space under the quasinorm $\|f\|_p= |(s_n(f))|_p=\big(\sum_n |(s_n(f))|^p\big)^{1/p}$, and a Banach space if $1\leq p < \infty$.


\medskip

We now prove that $\LH$ is a quite large ideal:

\begin{proposition}\label{prop:quitelarge}
$\mathcal{S}_p\subset \LH$ for every finite $p$.
\end{proposition}

\begin{proof}
We denote by $\FH$ the ideal of finite-rank operators on $\Hi$, using the subscript $p$ to indicate the norm in $\mathcal S_p$ for finite $p$.
The proof depends on the following result of \cite[Theorem 5.5(a)]{rmi}: 
For every $\phi\in\mathcal{Q}(\mathscr H)$ and every $0<p<\infty$ there is a homogeneous mapping $\Phi:  \FH\To\FH$ having the following properties:
\begin{itemize}
\item[(a)] If $x,y\in \mathscr H$, then $\|\Phi(y\otimes x)- y\otimes \phi(x)\|_p\leq C \|y\| \|x\|$.
\item[(b)] If $f$ has finite rank and $a\in \BH$, then $
\|\Phi(fa)-(\Phi f)a\|_p\leq C\|f\|_p\|a\|$
\end{itemize}
where $C$ depends only on $p$ and the quasilinearity constant of $\phi$.
The key point is that if $\phi$ and $\Phi$ are as before, then for every finite-rank $f$ the operator $\Phi(f)$ is a good approximation of the quasilinear map $\phi\circ f$:
Pick $x,y\in \mathscr H$ and consider the rank-one operator $y\otimes x$. Note that $f(y\otimes x)= y\otimes f(x)$ for any $f\in \BH$ and that the operator norm of $y\otimes x$ is $\|y\|  \|x\|$.
We have 
$$
\|\Phi( f(y\otimes x))- (\Phi f)(y\otimes x)\|_p\leq C\|f\|_p\|y\|\|x\|,
$$
by (b). But $(\Phi f)(y\otimes x)= y\otimes (\Phi f)(x)$ and 
$ \Phi( f(y\otimes x)) = \Phi(y\otimes f(x))$ so that 
$$
\| \Phi( f(y\otimes x))- y\otimes \phi(f(x))\|_p\leq C\|y\|\|f(x)\|,
$$
by (a). Combining,
$$
\| y\otimes (\Phi f)(x)- y\otimes \phi(f(x))\|_p=
\| y\| \|(\Phi f)(x)- y\otimes \phi(f(x))\|  \leq 
 C\|y\|\|f(x)\|+ C\|f\|_p\|y\|\|x\|
$$
Hence
$$
\|(\Phi f)(x)- \phi(f(x))\|\leq 
 2C\|f\|_p\|x\| \quad\implies \quad \|\Phi f-  \phi\circ f\|\leq 
 2C\|f\|_p.
 $$
Now assume one has  a twisted Hilbert space $\xymatrix{0\ar[r] & \mathscr K \ar[r]^\imath & \mathscr T \ar[r]^\pi & \mathscr H \ar[r] & 0}$\hspace{-4pt}.
There is no loss of generality if we assume that $ \mathscr K= \mathscr H$ and that  $\mathscr T= \mathscr H\oplus_\phi \mathscr H$ for some $\phi\in \QH$. Fix $p<\infty$ and let $\Phi:  \FH\To\FH$ be as before, so that  $\|\Phi f-  \phi\circ f\|\leq 
 2C\|f\|_p$ for every $f\in\FH$. If $f\in \BH$ has finite rank, then the operator $(\Phi f,f):\mathscr H\To \mathscr H\oplus_\phi \mathscr H$ defined by  $(\Phi f,f)(x)=((\Phi f)(x),f(x))$ is a lifting of $f$. Besides,
 $$
 \|((\Phi f)(x),f(x))\|_\phi=
 \| (\Phi f)(x)-\phi(f(x))\|+\|\phi(f(x))\|_H\leq 3C\|f\|_p\|x\|,
 $$
 that is, $\|(\Phi f,f)\|\leq 3C\|f\|_p$.
Finally, for general $f\in \mathcal{S}_p$ we can write $f=\sum_{n\geq 0}f_n$ with $f_n\in\FH$ and $\|f_n\|_p\leq 2^{-n}\|f\|_p$. If $F_n$ is a lifting of $f_n$ with $\|F_n\|\leq M\|f_n\|_p$ then $F=\sum_{n\geq 0} F_n$ is a lifting of $f$, which completes the proof.

 Note that $\|F\|$ may depend not only on the norms of the summands, but also on the Banach-Mazur distance between  $\mathscr H\oplus_\phi \mathscr H$ and its Banach envelope which, in turn, depends only on $Q(\phi)$.
\end{proof}

It is clear that an operator belongs to {\em some} $\mathcal{S}_p$ if and only if its singular numbers are $O(n^{-\alpha})$ for {\em some} $\alpha>1$, which does not necessarily agree with $1/p$. 
An ideal that contains all the Schatten classes and is commonly regarded as reasonably small is the Macaev ideal $\mathcal{S}_\omega$ of those operators whose singular values satisfy $\sum_n s_n/n<\infty$, with the obvious norm, see \cite{Mac, Dav}. This ideal appears naturally in a surprisingly wide variety of situations, problems, results and applications, as a quick internet search reveals.

\begin{question}
Is every operator in the Macaev ideal liftable?
\end{question}

We do not know if $\LH$ can be given a {\em complete} ideal norm. Here, the question seems to be if each liftable operator is {\em uniformly} liftable in the following sense:

\begin{definition}\label{def:ul}
An operator $f\in B(\mathscr H)$ is said to be uniformly liftable if there is a constant $M$ such that for every quasilinear map $\phi\in\QH$ there is a linear endomorphism of $\Hi$ such that $\|\phi\circ f-L\|\leq M Q(\phi)$. 
Analogously, $f$ is said to be uniformly extensible if there is a constant $M$ such that for every quasilinear map $\phi\in\QH$ there is a linear endomorphism of $\Hi$ such that $\|f\circ \phi-L\|\leq M Q(\phi)$. 
\end{definition}

Let us denote by $\mathcal{L}_{\text{unif}}$ and $\mathcal{E}_{\text{unif}}$ the spaces of uniformly liftable and uniformly extensible operators on the ground Hilbert space $\Hi$, respectively.
These become normed ideals when equipped with
\begin{equation}\label{eq:L E}
\|f\|_\mathcal{L}=\|f\|+  \sup_{Q(\phi)\leq 1} \inf_L \|\phi\circ f-L\|\quad\text{and}\quad
\|f\|_\mathcal{E}=\|f\|+  \sup_{Q(\phi)\leq 1} \inf_L \|f\circ\phi-L\|.
\end{equation}
The proof of Proposition~\ref{prop:quitelarge} shows that $\mathcal{L}_{\text{unif}}$ contains every $\mathcal S_p$ and the inclusion is continuous. 
One has (we omit the proofs):
\begin{itemize}
\item $\mathcal{L}_{\text{unif}}=\mathcal{E}_{\text{unif}}$, with equivalent norms.
\item $\mathcal{L}_{\text{unif}}$ is complete.
\item The norms defined by (\ref{eq:L E}) do not vary if one considers bounded quasilinear maps and bounded linear maps only.
\end{itemize}
Actually one can ``eliminate'' the linear map from the definitions because of the existence of a universal constant $C$ such that for every  $\Phi\in \QH$ one has
$$
\inf_L \|\Phi-L\|\leq C\sup_{x_i\in H}\left\{\int_0^1\Big\| \Phi\big(\sum_{i\leq n}r_i(t)x_i \big)- \sum_{i\leq n}r_i(t)\Phi(x_i)\Big\|\, dt  : \sum_{i\leq n}\|x_i\|^2 \leq 1 	\right\}.
$$
An immediate consequence of the completeness of $\mathcal{L}_{\text{unif}}$ and Proposition~\ref{prop:quitelarge} is that $\mathcal{L}_{\text{unif}}$ is strictly larger than $\bigcup_{p<\infty}\mathcal S_p$.

\section{Diagonal operators and centralizers}
An idea that could work to show that liftable implies uniformly liftable is to find a single quasilinear map $\phi$ able to test liftability in the sense that $f$ is liftable if (and only if) $\|\phi\circ f-L\|<\infty$ for some linear map $L$. While we do not know if such a test exists, it turns out that this path is practicable when working with {\em centralizers} and {\em multiplication} operators, that is, moving from the linear category to the category of $\ell_\infty$-modules. To avoid unnecessary complications we will not consider $\ell_\infty$-modules explicitly and we give a simplified definition of the notion of centralizer that suffices to work with sequence spaces. From now on, if $X$ is a (quasi) Banach sequence space we denote by $X^0$ the (usually dense) subspace of finitely supported elements of $X$.

\begin{definition}
A centralizer on $X$ is a homogenenous mapping $\Phi:X^0\To X$ satisfying the estimate $\|\Phi(ax)-a\Phi(x)\|\leq C\|a\|_\infty\|x\|$ for some $C$ and all $a\in\ell_\infty, x\in X^0$.
\end{definition}
We denote by $C(\Phi)$ the least constant for which the inequality holds and call it the centralizer constant of $\Phi$. The space of all centralizers on $X$ is denoted by $\mathcal{C}(X)$.

Centralizers are automatically quasilinear maps and actually one has $Q(\Phi)\leq M C(\Phi)$, where $M$ is a constant depending only (on the modulus of concavity of) $X$; see \cite[Lemma 4.2]{kalcom} or \cite[Proposition 3.1]{k-jfa}. The notion of a centralizer was invented by  Kalton isolating a crucial property that most {\em derivations} appearing in interpolation theory share. If you feel curious about why centralizers are not defined on the whole space $X$, take a look at \cite[Corollary 2]{nla}.

Let us present some important examples of centralizers.
Let $x:\N\To\mathbb C$ be a sequence converging to zero. The rank-sequence of $x$ is defined as
$$r_x(n)=\big{|}\{k\in\N: \text{either } |x(k)|> |x(n)| \text{ or } |x(k)|= |x(n)| \text{ and } k\leq n\}\big{|},$$
that is, $r_x(n)$ is the place that $|x(n)|$ occupies in the decreasing rearrangement of $|x|$. If $\varphi: \R^2_+\To\mathbb C$ is a Lipschitz function vanishing at the origin, then the map $\Phi:\ell_p^0\To\ell_p$ defined by
\begin{equation}\label{eq:phi}
\Phi(x)=x\:\varphi\left(\log\frac{\|x\|_p}{|x|}, \log r_x\right)
\end{equation}
is a centralizer; this is a specialization of \cite[Theorem 3.1]{kalcom}.
Actually these centralizers are {\em symmetric} in the sense that
$\Phi(x\circ\sigma)= \Phi(x)\circ\sigma$ when $\sigma$ is a permutation of the integers. Taking $\varphi(s,t)=s$ and $\varphi(s,t)=t$ one obtains the {\em Kalton-Peck} and {\em Kalton-alone} maps given by
\begin{equation}\label{eq:KP}
\Omega(x)= x\,\log\frac{\|x\|_p}{|x|} \qquad{\text{and}}\qquad \Gamma(x)= x\,\log r_x,
\end{equation}
respectively.
A {\em multiplication} operator on a sequence space $X$ is one of the form $x\longmapsto ax$ for some (necessarily bounded) sequence $a$. We denote by $a_{\bullet}$ the multiplication operator induced by $a$. Note that  $\| a_{\bullet}:X\To X\|=\|a\|_\infty$ and that multiplication operators leave $X^0$ invariant.

\begin{definition}
Let $0<p<\infty$.
A multiplication operator  $f$ on $\ell_p$ is liftable through centralizers if for every centralizer $\Phi:\ell_p^0\To \ell_p$ the restriction of  $f$ to $\ell_p^0$ has a lifting to $\ell_p\oplus_\Phi \ell_p^0$.
\end{definition}
The relevant picture is
$$
\xymatrixcolsep{4pc}
\xymatrixrowsep{2pc}
\xymatrix{
 &\ell_p\oplus_\Phi \ell_p^0 \ar[d]^\pi\\
  \ell_p^0 \ar[r]^{f} \ar[ur]^-F  & \ell_p^0 
}
$$
It should be clear that $f$ lifts through $\Phi$ if and only if there is a linear map $L:\ell_p^0\To \ell_p$ such that
$\|L-\Phi\circ f\|$ (equivalently, $\|L-f\circ\Phi\|$) is finite. In this case, and assuming $1\leq p<\infty$, the {\em new} linear map defined by
$$
\tilde{L}(x)=\int_{\Delta} v^{-1} L(vx)dv,
$$
where $dv$ denotes the Haar measure on the ``Cantor group'' $\Delta=\{\pm1\}^\mathbb N$ and the integral is taken in the Bochner sense, satisfies the same estimate as $L$ and commutes with every $v\in\Delta$ in the sense that $
\tilde{L}(vx)= v\tilde{L}(x)$. It quickly follows that $
\tilde{L}$ is implemented by some (perhaps unbounded) sequence $\lambda:\mathbb N\To \mathbb C$ in the sense that $L(x)=\lambda x$ for every $x\in\ell_p^0$. When $0<p<1$ one can simply take $\lambda(n)=(Le_n)(n)$.

We emphasize that we are really interested on this notion only for $p=2$.  
The reason for this sudden urge for generalization is:

\begin{lemma}Let $a$ be a bounded sequence. If the operator $a_\bullet:\ell_p\To \ell_p$ is liftable  through centralizers for some $0<p<\infty$, then so does for all $\,0<p<\infty$.
\end{lemma}

\begin{proof}
This is a consequence of the fact that centralizers can be ``shifted'' from one $\ell_p$ to any other. Precisely, if $0<q<p<\infty$ and $\Phi:\ell_p^0\To\ell_p$ is a centralizer, then the map $\Phi_q:\ell_q^0\To\ell_q$ defined by 
$\Phi_q(x)= u|x|^{q/s}\Phi( |x|^{q/p})$, where $q^{-1}=p^{-1}+s^{-1}$ and $f=u|x|$ is the polar decomposition, is a centralizer on $\ell_q$; {\em cf.} \cite[Lemma 5]{c}. Moreover, every centralizer on $\ell_q$ arises in this way, up to strong equivalence \cite[Corollary 3]{c}. Here, two homogeneous maps are said to be strongly equivalent
if their difference is bounded.

It therefore suffices to see that, given $a\in\ell_\infty, 0<q<p<\infty$ and a centralizer $\Phi:\ell_p^0\To\ell_p$ the multiplication operator induced by $a$ on $\ell_p^0$ lifts through $\Phi$ if and only if the corresponding operator on $\ell_q$  lifts through $\Phi_q$, with the same witness sequence.

($\Longrightarrow$) Let $\lambda$ be a sequence witnessing that  $a_\bullet$ lifts through $\Phi$, so that
$
\|\lambda x-a\Phi(x)\|_p\leq K\|x\|_p
$ for all $x\in \ell_p^0$.
Given a finitely supported $x\geq 0$ we can write $x=x^{q/s} x^{q/p}$, with $\|x\|_q= \|x^{q/s}\|_s \|x^{q/p}\|_p$ and, by H\"older inequality,
$$
\|\lambda x-a\Phi_q(x)\|_q= 
\|\lambda x^{q/s} x^{q/p} -ax^{q/s} \Phi(x^{q/p})\|_q\leq 
\| x^{q/s} \|_s \|\lambda x^{q/p} -a \Phi(x^{q/p})\|_p \leq K\|x\|_q.
$$

($\Longleftarrow$) We need the following additional property of $\Phi_q$: there is a constant $M$ so that for finitely supported $x,y$ one has $\|\Phi_q(xy)-x\Phi(y)\|_q\leq M\|x\|_s\|y\|_p$; see \cite[Proof of Lemma 5(b)]{c}. 
Now, if we assume $\|\lambda z-a\Phi(z)\|_q\leq M'\|z\|_q$ then for any finitely supported $x,y$ one has
$$
\|\lambda xy-ax\Phi(y)\|_q\leq M''\|x\|_s\|y\|_p
$$
and since $\|g\|_p=\sup_{\|f\|_s\leq 1}\|fg\|_q$ we are done.
\end{proof}

Experience dictates that, more often than not, the most efficient way to prove {\em things} is to connect them with some result of Nigel Kalton.
 A minimal extension of $\ell_1$ is an exact sequence of quasi Banach spaces
\begin{equation}\label{eq:minimal}
\begin{CD}
0@>>> \mathbb K @>\imath >> Z @>\pi>> \ell_1 @>>> 0
\end{CD}
\end{equation} 
Note that such a sequence splits if and only if $Z$ is (isomorphic to) a Banach space, so says the Hahn-Banach extension theorem. See \cite{rib,rob} and \cite[Section 4]{k} for the classical nontrivial examples.

In \cite[Section~8]{kaltconv}, Kalton studies multiplication operators on $\ell_1$ that lift to its minimal extensions and characterize them by the property that $d^\star_n\log n$ is bounded, where $d_n^\star$ is the decreasing rearrangement of the corresponding sequence, {\em cf.} \cite[Theorem 8.3]{kaltconv}. It turns out that these operators are exactly the liftable through centralizers:

\begin{proposition}
For a bounded sequence $d$ the following are equivalent:
\begin{itemize}
\item[(a)] $d_\bullet:\ell_2\To\ell_2$ liftable through centralizers.
\item[(b)] $d_\bullet:\ell_2\To\ell_2$ liftable through $\Omega$ or $\Gamma$; see $(\ref{eq:KP})$.
\item[(c)] $d^\star_n\log n$ is bounded.
\item[(d)] $d_\bullet:\ell_1\To\ell_1$  lifts to minimal extensions.
\end{itemize}
\end{proposition}

\begin{proof} 
We follow the string (a)$\implies$(b)$\implies$(c)$\implies$(d)$\implies$ (a). The first implication is trivial and the third one, by far the most difficult one, is contained in Kalton's result just mentioned.

(b)$\implies$(c)\: We write the proof for the Kalton-Peck centralizer $\Omega$. The proof for $\Gamma$ is easier.
 Assume $d$ is a decreasing sequence whose multiplication operator lifts through $\Omega$. If $d_n\log n$ is unbounded, then passing to a subsequence we may assume that $d_n\log n\To \infty$. We have $\Omega(de_i)=0$ and so every sequence witnessing that $d_\bullet$ lifts through $\Omega$ has to be bounded. It follows that $x\longmapsto \Omega(dx)$ is bounded on $\ell_2^0$ and this leads to a contradiction: For each $n\in\N$ consider the vector $s_n=\sum_{i\leq n}d_i^{-1}e_i$. One has
$$
\|s_n\|_2^2=\sum_{i\leq n} d_i^{-2},\,\text{ while }\,\|\Omega(ds_n)\|_2^2= \frac{n\log^2 n}{4} \quad \implies\quad \frac{\|\Omega(ds_n)\|_2^2}{\|s_n\|_2^2}=\frac{n\log^2 n}{4 \sum_{i\leq n} d_i^{-2}}\To \infty.
$$

(d)$\implies$(a)\: Take $\Phi\in\mathscr C(\ell_2)$ and define $\phi:\ell_1^0\To\K$ by
$$
\phi(x)=\langle 1_{\N}, u|x|^{1/2}\Phi \big( |x|^{1/2} \big)\rangle=
\langle 1_{\N}, \Phi_1 (x) \rangle,
$$
where $x=u|x|$ is the polar decomposition. This map is quasilinear since $\Phi_1$ is a centralizer. It is very easy to see that if $d_\bullet:\ell_1\To\ell_1$ lifts to the minimal extension $\mathbb K\oplus_\phi\ell_1^0$, then it lifts through $\Phi_1$ as well and, therefore, the corresponding operator on $\ell_2$ lifts through $\Phi$. Indeed, if $L:\ell_1\To\K$ is a linear map such that $\|L-\phi\circ d_\bullet\|<\infty$ and we consider the sequence $\lambda:\N\To\K$ defined by $\lambda(n)=L(e_n)$ so that $L(x)=\langle 1_{\N}, \lambda x \rangle$, one has
$$
|\langle 1_{\N}, \lambda x- \Phi_1 (dx) \rangle|= |L(x)-\phi(dx)|\leq M\|x\|_1\,\implies\,  |\langle 1_{\N}, \lambda  a x - \Phi_1 (da x) \rangle|
\leq M\|a\|_\infty \|x\|_1
$$
for every $a\in\ell_\infty$. By the centralizer property of $\Phi_1$ one also has
$$
|\langle a , \lambda x- \Phi_1 (dx) \rangle|= 
|\langle 1_{\N}, \lambda  a x- a \Phi_1 (dx) \rangle| \leq M'\|a\|_\infty \|x\|_1 \,\implies\,  \| \lambda x- \Phi_1 (dx)\|\leq M'\|x\|_1,
$$
that is, $\lambda$ witness that $d_\bullet$ lifts through $\Phi_1$ and the preceding  lemma applies.
\end{proof}

While the fact that $\Omega$ liftability implies liftability for all centralizers is hardly surprising ($\Omega$ is widely regarded as an ``extreme'' centralizer) the fact that $\Gamma$ shares this property is somewhat unexpected since it was not supposed to be so radical.
Actually the equivalence between $\Omega$-liftability and $\Gamma$-liftability
is not longer true for arbitrary operators: indeed, if $(h_n)$ is a disjoint normalized sequence in $\ell_2$ such that $\|h_n\|_\infty\To 0$, then the endomorphism sending $e_n$ to $h_n$ lifts from $\ell_2$ to $\ell_2\oplus_\Gamma\ell_2^0$  (by \cite[Proposition 2(b)]{factorization}, where $\Gamma$ is denoted by $\Upsilon$), but not to $\ell_2\oplus_\Omega\ell_2^0$ (by \cite[Theorem 6.4]{kp}).

Statement (c) shows that the multiplication operators liftable through centralizers form a Banach space (actually a Lorentz sequence space) under the norm $\|d\|=\sup_n d_n^\star\log(n+1)$.

We may summarize our results as follows: a sufficient condition for an operator to be liftable is that its singular numbers are $O(1/n^{\alpha})$ for some $\alpha>1$ and a necessary condition is that they are $O(1/\log n)$.

\medskip

 The content of this note naturally flows into the following:

\begin{question}
Is every operator whose singular numbers are bounded by $1/\log n$ liftable?
\end{question}

\end{document}